\documentclass[smallcondensed,envcountsect]{article}

\usepackage[noadjust]{cite}
\usepackage{marvosym}

\usepackage{amsmath}
\usepackage[T1]{fontenc}
\usepackage[utf8]{inputenc}
\usepackage[english]{babel}
\usepackage{graphicx}
\usepackage{amssymb}
\usepackage{flafter}
\usepackage{placeins}
\usepackage{pst-all}
\usepackage{multido}
\usepackage{mathrsfs}
\newtheorem{theorem}{Theorem}[section]
\newtheorem{definition}[theorem]{Definition}
\newtheorem{lemma}[theorem]{Lemma}
\newtheorem{proposition}[theorem]{Proposition}
\newtheorem{remark}[theorem]{Remark}
\newtheorem{corollary}[theorem]{Corollary}
\newenvironment{proof}[1][Proof]{\noindent\textbf{#1.} }{\ \rule{0.5em}{0.5em}}

\newcommand{\R}{\mathbb{R}}

\newcommand{\N}{\mathbb{N}}
\newcommand{\cl}[1]{\mbox{\rm cl}(#1)}
\newcommand{\tr}{{\rm tr}}
\newcommand{\qed}{\nobreak \ifvmode \relax \else
      \ifdim\lastskip<1.5em \hskip-\lastskip
      \hskip1.5em plus0em minus0.5em \fi \nobreak
      \vrule height0.75em width0.5em depth0.25em\fi}

\makeatletter

\begin{document}

\title{Minimal Failure Probability for Ceramic Design via Shape Control}
\author{Matthias Bolten \and Hanno Gottschalk \and Sebastian Schmitz}

\maketitle


\vspace{.3cm}

\begin{abstract}
We consider the probability of failure for components made of
brittle materials under one time application of a load, as
introduced by Weibull and Batdorf-Crosse. These models have been
applied to the design of ceramic heat shields of space shuttles
and to ceramic components of the combustion chamber in gas
turbines, for example. In this paper, we introduce the probability of failure
as an objective functional in shape optimization. We study the
convexity and the lower semi-continuity properties of such
objective functionals and prove the existence of optimal shapes in
the class of shapes with a uniform cone property. We also shortly
comment on shape derivatives and optimality conditions.
\end{abstract}



\section{Introduction}
Ceramics frequently is chosen to construct mechanical components.
Ceramics is temperature resistant and does not react with oxygen,
sulphur or hydrogen even at high temperatures. On the negative
side, the brittleness exposes ceramic structures to the risk of
spontaneous failure due to stress concentration at prefabricated
voids or inclusions. As the formation of such microcracks is
unavoidable in the sintering process and is stochastic by nature,
the failure under or the resistance to a given mechanical load is
a random event that occurs with a given failure probability.  This
was the insight by E. W. Weibull in his classical paper
\cite{Wei}.

In this article, we consider the probability of failure of a
mechanical component under a given load as objective functional in
shape optimization \cite{Ala,DZ,Epp,HM,SZ}. In the design of
ceramic components, models for the probability of failure have
been worked out for quite some time
\cite{Wei,BC,Heg,NMG,RBZ,RSEK,WD,Zie} and have found their way
into standard textbooks; see e.g.~\cite{MF}. The area of
application ranges from the design of heat shields in gas turbine
combustion chambers \cite{Heg,BHDRH,Hue} to those of the space
shuttle \cite{NMG}. Here, we follow the approach of \cite{BC} that
is also supported experimentally \cite{BFMS} in the case of small
flaw sizes. All these models have in common that the back reaction
of the cracks on the stress state is neglected.

While more detailed models are well studied in the materials
science community, see e.g.~\cite{Gam}, the models used here are
easy to implement on the basis of standard finite element software
by a simple post processing step involving some numerical
quadrature, only \cite{RBZ,RSEK}. Furthermore, they are comparatively
conservative in the number of parameters introduced by the models
and calibration procedures for these parameters are well studied
and standardized \cite{MF}. All these are important requirements
from industrial design processes.

We will prove that such probabilistic objective functionals, after
appropriate transformation to an equivalent problem, fulfil the
convexity requirements of Fujii \cite{Fuj} and thus are lower
semicontinuous in the weak topology of the Sobolev space
on a bounded constructed domain where the admissible shapes
share parts of their boundaries with this domain, cf.\ 
Figure~\ref{fig:omega_and_omega_hat}, and fulfill the uniform
cone property.

In the next step we apply lower semicontinuity to the problem of
shape optimization. Here the state equation is linear elasticity,
for simplicity~\cite{Ciar}. We 
conclude that there exists a shape
that has the lowest probability of failure among all
admissible shapes. This is the main result of our paper.

Although the existence result in this article is less general in
terms of the objective functionals than \cite{GoS}, it requires
much less restrictive boundary regularity assumptions and
technically follows a rather independent route.   For other work
on optimal design with the linear elasticity PDE as state
equation, see e.g.\ \cite{Ala,Epp,HM,ABFJ} and references
therein. These works however use objective functionals which
considerably differ in their design intention and mathematical
properties from what we consider here. This in particular applies
to the compliance functional, which is not directly related to the
failure of the component.

The paper is organised as follows. Section 2 essentially fixes
notation for the state equation and recalls well known facts from
linear elasticity. In Section 3 we give some background material
from linear fracture mechanics and introduce the Poisson point
process in order to derive failure probabilities. We derive
objective functionals that are minimal, if and only if the
probability of failure is minimal and which fit nicely into the
standard framework of shape optimization. Section 4 proves
convexity of the resulting objective functionals. In Section 5 we
apply the strategy of~\cite{Fuj} to conclude that optimally
reliable designs exist. Section 6 gives a short conclusion and an
outlook to shape derivatives and optimality conditions.

\section{\label{sect:Elasticity}Linear Elasticity in the Weak Formulation}
Let us start with our assumptions on the form of the ceramic component. We assume that the compact body $\Omega\subseteq \R^3$ is filled with the ceramic material. It is assumed that the boundary $\partial\Omega$ of $\Omega$ can be decomposed into three portions with not vanishing surface volume,
\begin{equation}
\label{sufaceDecomp}
\partial\Omega=\cl{\partial\Omega_D}\cup\cl{\partial\Omega_{N_\text{fixed}}}\cup
\cl{\partial\Omega_{N_\text{free}}}.
\end{equation}
The part is assumed to be fixed on $\partial\Omega_D$, the Dirichlet-Portion of the boundary. Furthermore, only on $\partial\Omega_{N_\text{fixed}}$ surface forces may act.  The free portion $\partial\Omega_{N_\text{free}}$ can be modified in order to optimally comply with the design objective `reliability', as we will explain below. For technical reasons that will become clear in Section \ref{sect:OptimalReliability}, the free boundary is assumed to be force-free. Furthermore, we assume that there is some bounded set $\hat\Omega\subseteq \R^ 3$ such that $\Omega\subseteq \hat \Omega$ for all admissible choices of the free portion of the boundary; see Figure \ref{fig:omega_and_omega_hat} for a two dimensional sketch.

Here we consider the linear elasticity partial differential equation as the state equation. Let $f:\hat \Omega\to\R^ 3$ be the the volume forces, e.g.\ gravitational or centrifugal forces, and $g:\hat\Omega\to\R^ 3$ the forces acting on the component's surface $\partial\Omega$, e.g.\ pressure or traction forces. With $u(\Omega):\Omega\to\R^3$ a twice differentiable function representing the displacement of $\Omega$ under the given loads, we consider the partial differential equation in the strong form
\begin{equation}
\label{strongEquation}
\begin{array}{ll}
-{\rm div} \sigma(u(\Omega))=f & \mbox{ on } \Omega\\
u(\Omega)=0&\mbox{ on } \partial \Omega_D\\
\sigma(u(\Omega))\hat n =g&\mbox{ on } \partial\Omega_{N_\text{fixed}}\\
\sigma(u(\Omega))\hat n =0&\mbox{ on } \partial\Omega_{N_\text{free}}
\end{array}
\end{equation}
with $\varepsilon(Du)=\frac{1}{2}(Du+Du^*)$ the elastic strain field, $\sigma(Du)=\lambda\, \tr(\varepsilon(Du))  I+2\mu\, \varepsilon(Du)$ the elastic stress field and $\mu,\lambda>0$ Lam\'e's constants. $Du$ stands for the Jacobi matrix of $u$ and $\hat n$ represents the outward directed unit normal vector field on $\partial\Omega$ provided that $\partial\Omega$ is piecewise differentiable. The index $N$ at the portions of the boundary here refers to natural boundary conditions\footnote{We note that in \cite{SZ} a different set of conditions are referred to as Neumann boundary conditions, hence the term natural boundary conditions in order to avoid confusion.}, which can be considered to be the proper generalization of Neumann boundary conditions to the case of systems of elliptic partial differential equations.

The theory of strong solutions of (\ref{strongEquation}) is quite involved \cite{Ciar}. Both for analytical and numerical reasons, the weak formulation of (\ref{strongEquation}) is generally preferred. We start stating the necessary regularity requirements on $\partial\Omega$ first.
We use the notation
\begin{equation}
\label{eqa:Cone}
C(\zeta,\theta,l):=\{x\in\mathbb{R}^3: |x|<l, x\cdot \zeta>|x|\cos(\theta)\}
\end{equation}
for the cone with height $l$, direction $\zeta$, and opening angle $\theta$.  We need the following definition:
\begin{definition}[\cite{Fuj,Chen}]
\label{defUC}
Let $\hat \Omega$ be a bounded open set in $\mathbb{R}^3$. For $\theta \in ]0,\pi/2[$, $l>0$, $r>0$, $2r \leq l$. By $\Pi(\theta,l,r)$ we denote the set of all subsets $\Omega$ of $\hat\Omega$ satisfying the cone property, i.e., for any $x \in \partial\Omega$ there exists a cone $C_x = C_x(\zeta_x,\theta,l)$, where $\zeta_x$ denotes a unit vector in $\mathbb{R}^3$, s.t.
\[
y + C_x \subset \Omega, \quad y \in B(x,r) \cap \Omega,
\]
where $B(x,r)$ is the open ball in $\mathbb{R}^n$ with radius $r$ centred at $x$.
\end{definition}

Based on the notion of the cone property and with Figure \ref{fig:omega_and_omega_hat} in mind, we can now define the admissible shapes or for our problem:
\begin{definition}
Let $\hat\Omega\subseteq \R^3$ be an fixed open set fulfilling the cone property with respect to some $\theta,l,r$ as in Definition \ref{defUC}. Let $\partial\Omega_D,\partial\Omega_{N_{\text{fixed}}}\subseteq \partial\hat\Omega$. Then,
 we define the admissible shapes
\[
\mathscr{O}^{\rm ad} := \{ \Omega \in \Pi(\theta,l,r)\ :\ \Omega \subset \widehat{\Omega}, \partial\Omega_D\subseteq \partial\Omega , \partial\Omega_{N_{\text{fixed}}}\subseteq \partial\Omega\}.
\]
Furthermore, for $V>0$, we define the admissible shapes with volume constraint $V$
\[
\mathscr{O}^{\rm ad}_V:=\{\Omega\in\mathscr{O}^{\rm ad}: |\Omega|=V\}.
\]
Here $|\Omega|:=\int_\Omega dx$ is the Lebesque volume of $\Omega$.
\end{definition}

Let $\Omega\in \Pi(\theta,l,r)$ and let $H^1(\Omega,\R^3)$ be the Sobolev space of $L^2(\Omega,\R^3)$ functions with square integrable first weak derivative, cf.\ \cite{Eva}. Then the restriction of $u\in H^1(\Omega,\R^3)$ to $\partial\Omega$ exists \cite{Ada} and we can define
\[
H^1_{\partial \Omega_D}(\Omega,\R^3)=\{u\in H^1(\Omega,\R^3): u\restriction_{\partial \Omega_D}=0\}.
\]
If we take the scalar product of both sides of (\ref{strongEquation}) with a test function $v\in H^1_{\partial\Omega_D}(\Omega,\R^3)$ and integrate over $\Omega$, we obtain the weak form of the elasticity PDE  on $\Omega$ with given loads $g\in L^2(\partial\Omega_{N_{\text{fixed}}},\R^3)$, $f\in L^2(\hat \Omega,\R^3)$ after application of the divergence theorem for Sobolev spaces \cite{Ada,Ciar}
\begin{equation}
\label{WeakEquation}
\mathscr{B}_\Omega(u(\Omega),v)=\int_\Omega f\cdot v\,dx+\int_{\partial\Omega_{N_\text{fixed}}} g\cdot v \,ds,~~~\forall v\in  H^1_{\partial \Omega_D}(\Omega,\R^3).
\end{equation}
The left-hand side of \eqref{WeakEquation} is given by
\begin{figure}
 \begin{center}
  \resizebox{0.4\textwidth}{!}{\input{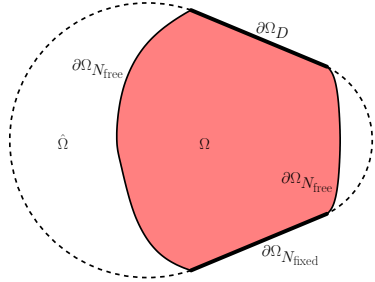}}
 \end{center}
 \caption{Domains $\Omega$ and $\hat{\Omega}$ represented in 2D, for simplicity.\label{fig:omega_and_omega_hat}}
\end{figure}

\begin{equation}
\label{BilinearForm}
\mathscr{B}_\Omega(u,v):=\int_\Omega \varepsilon(Du):\sigma(Dv) \, dx
\end{equation}
with $\varepsilon(Du)=\frac{1}{2}(Du+Du^*)$ the elastic strain field, $\sigma(Du)=\lambda\, \tr(\varepsilon(Du))  I+2\mu\, \varepsilon(Du)$ the elastic stress field and $\mu,\lambda>0$ Lam\'e's constants. $Du$ stands for the Jacobi matrix of $u$, and ${\rm tr}$ denotes the trace.

Furthermore, the elastic stress energy density $\varepsilon(Du):\sigma(Du)$  fulfils the following ellipticity condition
\begin{equation}\label{eq:elasticity_tensor_ellipticity}
\varepsilon(Du):\sigma(Du)\geq 2\mu \, \varepsilon(Du):\varepsilon(Du).
\end{equation}
From Korn's inequality for the displacement-traction problem \cite[Theorem 6.3-4]{Ciar}, we can now deduce the coercivity of $B_\Omega(.,.)$ on $H^ 1_{\partial\Omega_D}(\Omega,\R^ 3)$ and obtain the existence and uniqueness of the weak solution by the Lax-Milgram theorem \cite[Theorem 6.3-2]{Ciar}.

\begin{theorem}[Solution of the state equation, \cite{Ciar}]
\label{thm:stateEquation}
 Let $f\in L^2(\hat\Omega,\R^3)$, \linebreak$g \in L^2(\partial\Omega_{N_{\text{fixed}}},\R^3)$ and $\Omega\in\mathscr{O}^{\rm ad}$. Then, there exists a unique solution $u(\Omega)\in \linebreak H^1_{\partial\Omega_D}(\Omega,\R^3)$ to the linear elasticity PDE in its weak form (\ref{WeakEquation}).
\end{theorem}

\begin{figure}
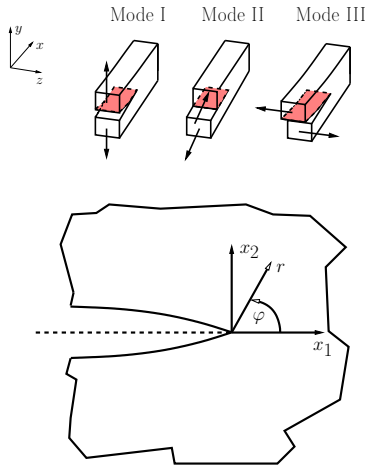

 \begin{center}
  \resizebox{0.4\textwidth}{!}{\input{Modus_I_II_III.pspdftex}}

\vspace{.5cm}

  \resizebox{0.35\textwidth}{!}{\input{Fracture_mechanics_r_phi_coords_1.pspdftex}}
 \end{center}
 \caption{Different modes of the loading (top). $r$-$\varphi$ coordinate system at the tip of the crack (bottom). \label{fig:modusI-III}}
\end{figure}

\section{Survival Probabilities from Linear Fracture Mechanics}
This section is devoted to the derivation of objective functions based on solutions to the state equation, linear fracture mechanics and Weibull's analysis of the stochastic nature of the ultimate strength of brittle material \cite{Wei}.

 Let us first recall some elements of the classical engineering analysis of spontaneous failure of mechanical components from brittle material under given mechanical loads. In linear fracture mechanics, the three dimensional stress field close to a crack in a two dimensional plane close to the tip of the crack is of the form
\begin{equation}
\label{multiAx}
\sigma=\frac{1}{\sqrt{2\pi r}}\{K_I\tilde\sigma^I(\varphi)+K_{II}\tilde \sigma^{II}(\varphi)+K_{III}\tilde \sigma^{III}(\varphi)\}+\mbox{regular terms},
\end{equation}
where the detailed form of the shape functions $\tilde \sigma^\#(\varphi)$ is determined by complex analysis, \cite[chapter 4]{GS}. Here $r$ is the distance to the crack front and $\varphi$ the angle of the shortest connection point considered to the crack front with the crack plane. The $K$-factors -- also called stress intensity factors -- depend on the amount and the mode of the loading, cf.\ Figure~\ref{fig:modusI-III}, and the geometry of the crack. While no simple stochastic models for the crack shapes exist, it is customary to use the concept of equivalent circular disk shaped reflectors to register crack sizes with non destructive ultra sonic measurements. Here we follow this approach and restrict ourselves to circular `penny shaped' cracks. Considering e.g. the tensile loading $\sigma_n$ in a normal direction of the stress plane and the crack geometry circular with radius $a$, one obtains
\begin{equation}
\label{Kfactor}
K_I:=\frac{2}{\pi}\sigma_n \sqrt{\pi a}.
\end{equation}
Failure occurs in the ceramic component if $\sigma_n$ is positive and is large enough such that the stress intensity $K_I$ exceeds a critical value $K_{Ic}$. Typical $K_{Ic}$ values for ceramics that are measured in mechanical tests are $(3 \mbox{ to } 16) \, [{\rm MPa\sqrt{m}}]$.  Apparently, in the case of compressive loads, i.e., $\sigma_n<0$, no failure will occur no matter what the size $a$ of the crack is. We note that it would be straight forward to incorporate more complex flaw geometries in the framework of this article, e.g.\ for elliptic shapes $K_I$ is modified with a factor $1-\sqrt{1-c^2}$,  where $0<c\leq 1$ the quotient between the length of the principal axes. The consideration of surface cracks (eg.\ due to manufacturing) will require the more involved analysis of \cite{GoS}.

Next step is the passage to stress fields with arbitrary orientation w.r.t. the crack plane, see (\ref{multiAx}). A large number of solutions has been proposed to the extension of the concept of critical $K$ factors to the multi axial case \cite{BC,Eva,GS,Wei,Zie}. Experimental evidence \cite{BFMS} indicates that for microscopic or mesoscopic initial flaws the shear stress influence to the strength of a ceramic component is negligible. We therefore follow \cite{Heg,Zie} and set
\begin{equation}
\label{normalStress}
\sigma_n:=(n\cdot\sigma(Du)n)^+=\max\{n\cdot\sigma(Du)n,0\},
\end{equation}
retaining the failure criterion $K_I(a,\sigma_n(x))>K_{Ic}$ at the location $x\in\Omega$ of a crack with radius $a$.

The probabilistic model of flaw distributions is a marked Poisson point process (PPP) with the mark space given by $S^2\times \R_+$. Here $S^2$ stands for the flaw orientation described by the normal $n$ and $\R_+$ parametrizes the flaw radius $a$. Assumptions that lead to the PPP model are:
\begin{itemize}
\item Flaws are uniformly distributed over the volume $\Omega$ of the component with an average number $z>0$ of flaws per unit volume;
\item Two flaws can always be distinguished either by their orientation, size or by their location;
\item Orientations are uniformly distributed over $S^2$ and are independent of the flaw location;
\item The distribution of the flaw radius is independent of location and orientation of the crack;
\item The number of flaws in given, non intersecting volumes $A_1,\ldots,A_n\subseteq \Omega$ are statistically independent of each other.
\end{itemize}
If these assumptions are a good approximation to reality, the following mathematical model is adequate and essentially fixed by these assumptions, confer \cite[Corollary 4.7]{Kal}:
\begin{definition}
\label{DefPPP}\rm
Let $\mathscr{M}=\Omega\times S^2\times\R_+$ be the crack configuration space endowed with the sigma algebra $\mathscr{A}(\mathscr{M})$ defined as the Borel sigma algebra on $\mathscr{M}$. Let furthermore $\nu$ be the Radon measure on  $\mathscr{A}(\mathscr{M})$ which is given by
\begin{equation}
\label{M-Measure}
\nu=dx\restriction_\Omega\otimes\frac{dn}{4\pi}\otimes \rho.
\end{equation}
Here $dx$ is the Lebesgue measure on $\R^3$, $dn$ the surface measure on $S^2$ and $\rho$ a positive Radon measure on $(\R_+,{\cal B}(\R_+))$ such that $\rho([c,d])$ fixes the density (number per volume) of cracks with radius $a$, $c\leq a\leq b$. A natural assumption is that only finitely many cracks with a radius $a$ above some finite limit can be contained in a given volume, i.e. $\rho([c,\infty[)<\infty$ $\forall c>0$.

The Poisson point process on the crack configuration space $\mathscr{M}$ with intensity measure $\nu$ is a mapping $N:\mathscr{E}\times \mathscr{A}(\mathscr{M})\to\mathbb{N}_0$, where $\mathscr{E}$ is the set of some probability space $(\mathscr{E},\mathscr{A},P)$ such that the following conditions hold:
\begin{itemize}
\item[(i)] $\forall A\in \mathscr{A}(\mathscr{M})$, $N(A)=N(.,A):(\mathscr{E},\mathscr{A},P)\to(\N_0,\mathscr{P}(\N_0))$ is a (counting) random variable;
\item[(ii)] $\forall \omega\in \mathscr{E}$, $N(\omega,.): \mathscr{A}(\mathscr{M})\to \N_0\subseteq \bar \R_+$ is a sigma finite measure;
\item[(iii)] $\forall n\in\mathbb{N}, A_1,\ldots,A_n\in\mathscr{A}(\mathscr{M})$ mutually disjoint, the random variables $N(A_1),\linebreak\ldots,N(A_n)$ are independent;
\item[(iv)] $\forall A\in\mathscr{A}(\mathscr{M})$ such that $\nu(A)<\infty$, $N(A)$ is Poisson distributed with mean $\nu(A)$, $N(A)\sim{\rm Po}(\nu(A))$, i.e.,
\begin{equation}
P(N(A)=n)=e^{-\nu(A)}\frac{\nu(A)^n}{n!}.
\end{equation}
\end{itemize}
\end{definition}

 Items i) and ii) are the definition of a general point process, iii) is needed for a general PPP on $\mathscr{M}$ and iv) fixes its distribution  \cite{Kal}.

\begin{lemma} Let $u\in H^1(\Omega,\R^3)$ be given, then
\begin{equation}
\label{critSet}
A_c=A_c(\Omega,Du)=\left\{(x,n,a)\in \mathscr{M}: K_I\left(a,(n\cdot\sigma(Du(x))n)^+\right)> K_{Ic}\right\}\in \mathscr{A}(\mathscr{M}).
\end{equation}
\end{lemma}
\begin{proof} $Du\in L^2(\Omega,\R^{3\times 3})$ is Borel measurable and so is $\sigma_n=(n\cdot \sigma(Du)n)^+$. Thus the set of critical crack configurations given $\sigma(Du)$ is measurable as the pre-image of the interval $[K_{Ic},\infty)$ under the Borel measurable function
\[
\mathscr{M}\ni(x,n,a)\to K_I\left(a,(n\cdot\sigma(Du(x))n)^+\right)\in\R_+.
\]

\end{proof}

Adopting the point of view that the component fails if there is any crack with configuration in the critical set $A_c(\Omega,Du)$, hence $N(A_c(\Omega,Du))>0$, we obtain the following definition for the survival probability:

\begin{definition}
\label{defSurvProb}\rm
The survival probability of the component $\Omega$, given the displacement field $u\in H^1(\Omega,\R^3)$, is
\begin{equation}
\label{SurvProb}
p_s(\Omega|Du)=P(N(A_c(\Omega,Du))=0)=\exp\{-\nu(A_c(\Omega,Du))\}.
\end{equation}
\end{definition}

A more explicit representation of $\nu(A_c(\Omega,Du))$ can be found with the help of the cumulative crack size function $\Phi(s):=\rho(]s,\infty[)$ of the crack radius; see also \cite{Heg,Zie}:
\begin{lemma}
\label{lemHExplizit}
Let $u\in H^1(\Omega,\R^3)$; then $\nu(A_c(\Omega,Du))=\int_\Omega h(Du) \, dx$ with
\begin{equation}
\label{hexplizit}
h(q):=\frac{1}{4\pi}\int_{S^2}\Phi\left(\frac{\pi}{4}\left(\frac{K_{Ic}}{(n\cdot q\,n)^+}\right)^2\right)\,dn~~q\in\R^{3\times 3}.
\end{equation}
\end{lemma}
\begin{proof} By Fubini's theorem for positive functions with $\chi(B)$ the characteristic function of the set $B$,
\begin{eqnarray*}
\nu(A_c(\Omega,Du))&=&\frac{1}{4\pi}\int_\Omega\int_{S^2}\int_{\R_+} \chi(\{K_I(a,(n\cdot Du(x)\,n)^+)>K_{Ic}\})\, d\rho(a)\,dn\,dx\\
                    &=&\frac{1}{4\pi}\int_\Omega\int_{S^2} \rho\left(a>\frac{\pi}{4}\left(\frac{K_{Ic}}{(n\cdot Du\,n)^+}\right)^2\right) dn\,dx\\
                    &=& \frac{1}{4\pi}\int_\Omega\int_{S^2}\Phi\left(\frac{\pi}{4}\left(\frac{K_{Ic}}{(n\cdot Du\,n)^+}\right)^2\right) dn\,dx.
\end{eqnarray*}

\end{proof}

For later use, we prove the following:
\begin{lemma}
\label{lem:hconti}
The function $h(q)$ introduced in \eqref{hexplizit} depends continuously of $q$.
\end{lemma}
\begin{proof}
We first note that for $q_l\to q$ with $(n\cdot qn)^+>0$,
$$
\Phi\left(\frac{\pi}{4}\left(\frac{K_{Ic}}{(n\cdot q_ln)^+}\right)^2\right)\to \Phi\left(\frac{\pi}{4}\left(\frac{K_{Ic}}{(n\cdot qn)^+}\right)^2\right)
$$
$dn$ almost everywhere, since by upper continuity of the radon measure $\rho$ on sets of finite measure $\Phi(\kappa)$ has at most countably many non continuity points. Let us now assume that $(n\cdot qn)^+=0$. In this case $\frac{\pi}{4}(\frac{K_{Ic}}{(n\cdot q_ln)^+})^2\to\infty$, and thus
$$
\Phi\left(\frac{\pi}{4}\left(\frac{K_{Ic}}{(n\cdot q_ln)^+}\right)^2\right)\to 0=\Phi\left(\frac{\pi}{4}\left(\frac{K_{Ic}}{(n\cdot qn)^+}\right)^2\right)
$$
again by upper continuity and additivity of $\rho$. Furthermore, the integrand in the $S^2$ integral defining $h(q)$ by additivity of $\rho$ is uniformly bounded by
$$
 \Phi\left(\frac{\pi}{4}\left(\frac{K_{Ic}}{\sup_{n\in S^2,l\in\N}(n\cdot q_ln)^+}\right)^2\right)<\infty.$$
The assertion of the lemma thus follows from  Lebesgue's theorem of dominated convergence.    
\end{proof}

Apparently, in mechanical design we want to maximise the survival probability $p_s(\Omega|Du(\Omega))$ in the shape control variable $\Omega\in\mathscr{O}^{\rm ad}$ under the PDE constraint (\ref{WeakEquation}). Obviously, by Definition \ref{defSurvProb}, this is equivalent to the minimization of $\nu(A_c(\Omega,Du))$.  Using Lemma \ref{lemHExplizit} we can reformulate this into the following standard PDE constraint minimization problem:

\begin{definition}
\label{defSOProblem}
The problem of optimal reliability for a ceramic component $\Omega\in\mathscr{O}^{\rm ad}$ under given volume load $f\in L^2(\Hat\Omega,\R^3)$ and surface load $g\in L^2(\partial\Omega_{N_\text{fixed}},\R^3)$ is defined as the following shape optimization problem:
\begin{equation}
\label{eqa:shape_optimization_problem}
 \begin{array}{l}
 \mbox{Find }\Omega^ *\in\mathscr{O}^{\rm ad.}\mbox { s.t. } J(\Omega^ *,u(\Omega^ *))\leq J(\Omega,u(\Omega))~ \forall \Omega\in\mathscr{\cal O}^{\rm ad}\\
 u=u(\Omega)\mbox{ solves  the state equation (\ref{WeakEquation}) }\\
 J(\Omega,u(\Omega)):=\int_\Omega h(Du(\Omega))\, dx \mbox{ with }  h\mbox{ defined in Lemma \ref{lemHExplizit}}.
 \end{array}
\end{equation}
Here $\mathscr{O}^{\rm ad}$ can be replaced with $\mathscr{O}^{\rm ad}_V$ for the volume constrained optimal reliability problem as long as $0<V<|\widehat\Omega|$ such that $\mathscr{O}^{\rm ad}_V\not=\emptyset$.
\end{definition}

From a shape optimization perspective, the objective functional $J(\Omega,u(\Omega))$ has attractive properties as:
\begin{itemize}
\item It has a clear material science derivation and a proven record of industrial application \cite{BHDRH,Hue,Zie};
\item It permits to show the existence of optimal shapes by its convexity properties \cite{Fuj};
\item One can prove the existence of the shape derivatives $d\mathscr{J}(\Omega,V)$ under infinitesimal transformations generated by a vector field $V$ in the sense of \cite{SZ}, confer the forthcoming work \cite{GKS}.
\end{itemize}

\section{Convexity of the Objective Functional}
Fujii showed \cite{Fuj} that any objective functional $J(\Omega,u)=\int_\Omega h(D u) \, dx$ with convex, positive function $h:\R^3\to\R_+$ is lower semcontinuous in the weak $H^1(\Omega, \R)$ topology for scalar $u\in H^1(\Omega,\R)$. As lower semicontinuity is an essential ingredient to existence proofs for optimal shapes, we now look for conditions on the crack radius distribution that will ensure the convexity of the function $h(q)$.

\begin{definition}
\label{defConv}
A crack size measure $\rho$ has the non decreasing stress hazard property, iff the function $H:\R_+\to \R$ defined as
\begin{equation}
H(\kappa):=\Phi\left(\frac{1}{\kappa^ 2}\right)=\rho\left(\left]\frac{1}{\kappa^ 2},\infty\right[\,\right)
\end{equation}
is convex in $\kappa$.
\end{definition}

To understand the physical content of Definition \ref{defConv}, let us take the simplifying assumption that the stress state $\sigma$ is homogeneous and tensile.  We only consider such cracks with crack plane normal  $n$ in a small neighbourhood $U(\bar n)\subseteq S^ 2$ of $\bar n\in S^ 2$. Approximately, we can replace all crack orientations in $U(\bar n)$ by $\bar n$ itself. Let $\sigma_{\bar n}=\bar n\sigma \cdot \bar n$, then the probability for the absence of failure due to a crack with orientation in $U(\bar n)$ at the stress level $\sigma_{\bar n}$ approximately is
\begin{equation}
P(S(\bar n)>\sigma_{\bar n}))=P(N(A_c(\sigma_{\bar n},\bar n))=0)\approx \exp\left\{ -|\Omega||U(\bar n)| H\left(\frac{2}{\sqrt{\pi}}\frac{\sigma_{\bar n}}{K_{Ic}}\right) \right\}.
\end{equation}
Here $A(\sigma_{\bar n},\bar n)=A_c(\Omega,Du)\cap \R^3\times U(\bar n)\times [0,\infty[$ stands for the critical set associated to a stress state $\sigma$ with $\sigma_{\bar n}=\bar n\sigma \bar n$.
Therefore, up to a rescaling of the stress and a positive pre-factor, $H(\kappa)$ is the cumulative hazard function of the random strength $S(\bar n)$ for the hazard of rupture due to a crack with orientation in $U(\bar n)$ \cite{EM}. If we suppose that $H(\kappa)$ is differentiable,  convexity of $H(\kappa)$ in $\kappa$ is equivalent to a non decreasing hazard rate $\tilde h(\kappa)=H'(\kappa)$ in $\sigma_{\bar n}$:
\begin{equation}
\frac{2|\Omega||U(\bar n)|}{\sqrt{\pi}K_{Ic}}\tilde h\left(\frac{2}{\sqrt{\pi}}\frac{\sigma_{\bar n}}{K_{Ic}}\right)\approx\lim_{\Delta\searrow 0} \frac{P(S(\bar n)<\sigma_{\bar n}+\Delta|S(\bar n)>\sigma_{\bar n})}{\Delta},
\end{equation}
cf. \cite{EM}. Thus, the non decreasing stress hazard property implies that the risk of failure because of a crack with orientation in $U(\bar n)$ due to augmentation of the stress $\sigma_{\bar n}$ by an amount $\Delta$, provided the component sustained the stress $\sigma_{\bar n}$, does not decrease with the stress level $\sigma_{\bar n}$.  This kind of behaviour can be expected from a wide range of materials.

\begin{proposition}
\label{propCov}
Suppose that the crack size measure $\rho$ fulfils the non decreasing stress hazard property. Then the function $h$ defined in Lemma \ref{lemHExplizit} is convex.
\end{proposition}
\begin{proof}
Let $q_1,q_2\in \R^{3\times 3}$ and $t\in]0,1[$. With $\kappa_j:=\frac{2(n\cdot q_j\, n)}{\sqrt{\pi}{K_{Ic}}}\not=0$, $j=1,2$, we get from the convexity of $H(\kappa)$ defined in Definition \ref{defConv} that
 \begin{eqnarray*}
 \Phi\left(\frac{\pi}{4}\left(\frac{K_{Ic}}{(n\cdot (tq_1+(1-t)q_2)\,n)^+}\right)^2\right)&=& H((t\kappa_1+(1-t)\kappa_2))\\
 &\leq&tH(\kappa_1)+(1-t)H(\kappa_2)\\
 =t\Phi\left(\frac{\pi}{4}\left(\frac{K_{Ic}}{(n\cdot q_1\,n)^+}\right)^2\right)
 &+&(1-t)\Phi\left(\frac{\pi}{4}\left(\frac{K_{Ic}}{(n\cdot q_2\,n)^+}\right)^2\right).
 \end{eqnarray*}
The case that involves one or two $\kappa_j=0$ is trivial as the right and left-hand side are equal in this case. Integration of this inequality in $n$ over $S^2$ then yields convexity of $h$. 
\end{proof}

\begin{proposition}
\label{propDensityConvex}
Suppose that the crack size density measure $\rho$ is absolutely continuous w.r.t.\ the Lebsgue measure $da$, $d\rho(a)=\varrho(a)da$, $\varrho(a)>0$ for $a\in\R_+$. We furthermore assume that $\alpha(a):=-\log \varrho(a)$ is differentiable on $\R^+$ and
\begin{equation}
\label{convCond}
\alpha'(a)\geq\frac{3}{2}\frac{1}{a}~~\forall a>0.
\end{equation}
Then $\rho$ fulfils the non decreasing stress hazard property and the function $h$ defined in Lemma \ref{lemHExplizit} is convex.
\end{proposition}
\begin{proof} Using the function $H(\kappa):=\Phi\left(\frac{1}{(\kappa^+)^2}\right)$,  the natural extension for $\kappa\leq 0$ is $H(\kappa)=0$. This  corresponds to $\lim_{a\to\infty} \Phi(a)=\rho(]a,\infty[)=0$ by upper continuity of the Radon measure $\rho$ for sequences of decreasing sets with finite measure.

 Note that $H(\kappa)$ is continuous and second-order differentiable for $\kappa\in\R^*=\R\setminus\{0\}$. Thus, to show convexity, it suffices that $H''(\kappa)\geq 0$ $\forall \kappa \in\R\setminus\{0\}$. This is trivially true for $\kappa<0$ as then $H''(\kappa)=0$. Let us now investigate the case $\kappa>0$. We get
$$
H''(\kappa)=-4\varrho'\left(\frac{1}{\kappa^2}\right)\frac{1}{\kappa^6}-6\varrho\left(\frac{1}{\kappa^2}\right)\frac{1}{\kappa^4}\stackrel{!}{>}0.
$$
This is equivalent to
$$
-\frac{\varrho'\left(\frac{1}{\kappa^2}\right)}{\varrho\left(\frac{1}{\kappa^2}\right)}=\alpha'\left(\frac{1}{\kappa^2}\right)\stackrel{!}{>}\frac{3}{2}\kappa^2
$$
which holds by the assumption (\ref{convCond}) using the substitution $a=\frac{1}{\kappa^2}$.

Now apply Proposition \ref{propCov} to show convexity of $h$.
 \end{proof}

Condition (\ref{convCond}) restricts the tail behaviour of the $a$-density $\varrho(a)$ to a decrease at least as fast as ${\rm const.}\times a^{-\beta}$ for $a\to\infty$ with $\beta\geq 3/2$ as $\alpha\geq {\rm const.}+\beta \log(a)$.

Assuming an algebraic scaling for $\varrho(a)$, we can make contact with the classical Weibull type objective functionals \cite{Heg,Wei,Zie}.

\begin{proposition}
\label{PropWeibull}
Let $u\in H^1(\Omega,\R^3)$ and $\beta\geq\frac{3}{2}$ be given such that
\begin{equation}
\alpha(a):=\alpha_0+\beta \log(a), ~\mbox{i.e.}~\varrho(a)=e^{-\alpha_0} a^{-\beta}~~~\forall a>0.
\end{equation}
Then
\begin{equation}\label{eq:Jbar}
J(\Omega,u)=\frac{1}{4\pi}\int_{\Omega}\int_{S^2}\left(\frac{n\cdot\sigma(Du)n}{\sigma_0}\right)^m\, dn\, dx,
\end{equation}
with $m=2(\beta-1)\geq 1$ and
\begin{equation}
\label{sigmaZero}
\sigma_0=e^{-\alpha_0/2(\beta-1)}(\beta-1)^{1/2(\beta-1)}\sqrt{\frac{4}{\pi}}K_{Ic}.
\end{equation}
\end{proposition}
\begin{proof} We have
$
\Phi(a)=\frac{e^{\alpha_0}}{(\beta-1)} a^{-(\beta-1)}
$.
One obtains
$$
\Phi\left(\frac{\pi}{4}\left(\frac{K_{Ic}}{(n\cdot \sigma(Du)n)^+}\right)^2\right)=\left(\frac{(n\cdot\sigma(Du)n)^+}{e^{-\alpha_0/2(\beta-1)}(\beta-1)^{1/(2(\beta-1)}\sqrt{\frac{4}{\pi}}K_{Ic}}\right)^{2(\beta-1)}.
$$

\end{proof}

\begin{remark}{\rm 
The above objective functional was introduced by Weibull in \cite{Wei} based on statistical evidence. Our derivation from the distribution of crack sizes is pretty standard in material science; see e.g.\ \cite[Chapter 5]{MF}.
For a different derivation of the same functional from the large sample limit of extreme value theory are applied along with some approximations that can be controlled numerically to a reasonable extent confer \cite{BC,Heg,Zie}.
}\end{remark}

\begin{remark}\rm Typical experimental values of $m$ range from 5 to 25; see \cite{Wei}. In particular the assumptions of Proposition \ref{propDensityConvex} do not rule out the cases of physical interest. Note that the large $m$ limit is deterministic.
\end{remark}

\begin{remark}\rm The dimensional mismatch between $\sigma_0$ and the stress intensity $K_{Ic}$ in equation (\ref{sigmaZero}) is explained by the fact that $\Phi(a)$ is a functional of a dimensional quantity $a$. Understanding $\Phi$ as a function of a numerical value, we need to introduce a length scale $a_0={\rm[m]}$ and consider $\Phi(a/a_0)$, which divides $K_{1c}$ by $\sqrt{a_0}$.
\end{remark}
\begin{corollary}\label{cor:Weibull_local_failure_intensity_function}
Let the Weibull local failure intensity function $h_W:\R^{3\times 3}_s\to \R_+$ be defined as
\begin{equation}
h_W(q):=\frac{1}{4\pi}\int_{S^2}\left(\frac{(n\cdot q\, n)^+}{\sigma_0}\right)^m\, dn. \label{eq:Weibull_local_failure_intensity_function}
\end{equation}
Then $h_W$ is convex for $m\geq 1$ and continuous.
\end{corollary}
\begin{proof}
Apply Propositions \ref{propCov}, \ref{propDensityConvex} and \ref{PropWeibull}. 
\end{proof}

\section{\label{sect:OptimalReliability}Shapes with Optimal Survival Probability}
Having the results of the previous section at hand, we can now show the existence of shapes solving the shape optimization problem given in Definition \ref{defSOProblem}  -- hence solutions with optimal survival property.

As we will deal with problems involving mixed boundary conditions, we need an appropriate extension operator.
\begin{theorem}[Theorem II.1 in \cite{Chen}]\label{thm:Chenais_ii.1}
Let $\theta, l, r \in \mathbb{R}$ s.t.\ $\theta \in ]0,\pi/2[$ and $2r \leq l$ and let $n \in \mathbb{N}$. There exists a constant $K(\theta,l,r)$ depending on $\Omega \in \Pi(\theta,l,r)$ through $\theta, h, r$, only, and s.t.\ for all $\Omega \in \Pi(\theta,l,r)$ there exists a linear and continuous extension operator $p_\Omega : H^n(\Omega,\R^3) \rightarrow H^n(\mathbb{R}^3,\R^3)$, s.t. $p_\Omega u(x) = u(x)$ for all $x \in \Omega$, with
\[
\|p_\Omega\| \leq K(\theta,l,r).
\]
\end{theorem}
\begin{proof}
See proof of Theorem~II.1 in \cite{Chen}. 
\end{proof}

Further, we need the following result.
\begin{lemma}[\cite{Fuj}]\label{lem:pi_theta_h_r_is_relative_compact_and_closed}
The class $\mathscr{O}^{\rm ad}$ of domains is relatively compact and is closed with respect to the strong $L^2(\widehat \Omega)$ topology, i.e. the metric topology from $d(\Omega,\Omega')=\| \chi(\Omega)-\chi(\Omega')\|_{L^2(\widehat\Omega)}$. Here $\chi(\Omega)$ stands for the characteristic function of the set $\Omega$.

The same applies to the volume constraint sets $\mathscr{O}_V^{\rm ad}$, provided $0<V<|\widehat \Omega|$.
\end{lemma}
\begin{proof}
Theorem~III.1 in \cite{Chen} states that $\Pi(\theta,l,r)$ is relative compact, Theorem~III.2 in \cite{Chen} shows that it is closed. $\mathscr{O}^{\rm ad}$ obviously is closed in $\Pi(\theta,l,r)$. The second statement follows from the fact that $\mathscr{O}_V^{\rm ad}\subseteq \mathscr{O}^{\rm ad}$ is closed in the $L^2(\widehat\Omega)$-topology. 
\end{proof}

The main tool for showing the existence of optimal shape is the following theorem.
\begin{theorem}\label{thm:fuji_thm2.1}
Let $h$ be continuous, non-negative, and convex. Assume that for \linebreak $\{\Omega_n\} \subset \Pi(\theta,l,r)$ we have
\[
\Omega_n \rightarrow \Omega, \quad \text{a.e. in} \; \widehat \Omega,
\]
i.e., the characteristic functions of $\Omega_n$ converge to the characteristic function of $\Omega$ in $L^2(\widehat\Omega)$, and that for the extension $\tilde{u}_n = p_\Omega(u_n)$ of $u_n \in H^1(\Omega_n,\R^3)$ we have
\[
\tilde{u}_n \rightharpoonup \tilde{u}, \quad \text{in} \; H^1(\widehat \Omega,\R^3),
\]
where $\tilde{u} = p_\Omega(u)$. Then, the following inequality holds:
\[
\int\limits_\Omega h(Du(x)) dx \leq \liminf\limits_{n\rightarrow\infty} \int\limits_{\Omega_n} h(Du_n(x)) dx.
\]
\end{theorem}
\begin{proof}
The proof of Theorem 2.1 in \cite{Fuj} extends without modifications from scalar $u$ to vector valued $u$. 
\end{proof}

In order to apply this theorem, we have have to show that an arbitrary sequence $\{\Omega_n,u_n\}$ of admissible domains and solutions has a subsequence that  converges.

\begin{lemma}\label{lem:convergence_of_domain_and_solution}
Let $\partial\Omega_D$ and $\partial\Omega_N$ be defined as above and let $\{\Omega_n,u_n\}_{n=1}^\infty$ an arbitrary sequence of admissible domains and their corresponding solutions, i.e., $\Omega_n \in \mathscr{O}^{\rm ad}$ and $u_n=u(\Omega_n)$ solves~\eqref{WeakEquation} in the domain $\Omega_n$. Then one can find its subsequence also denoted by the pair $(\Omega_n,u_n)$ and elements $\Omega \in \Pi(\theta,l,r)$ and $u \in (H^1(\mathbb{R}^3),\R^3)$ such that
\[
\Omega_n \rightarrow \Omega,~~\mbox{ and }~~
\tilde{u}_n \rightharpoonup \tilde{u},
\]
where $\tilde{u}_n$ and $\tilde{u}$ are the extensions of $u_n$ and $u$ to $\mathbb{R}^n$ and $u$ solves \eqref{WeakEquation} in $\Omega$. The same holds true, if $\mathscr{O}^{\rm ad}$ is replaced with the volume constrained sets $\mathscr{O}_V^{\rm ad}$.
\end{lemma}
\begin{proof}
We define the set of admissible displacements as
\[
\mathbb{V}(\Omega) := \{ v \in H^1(\Omega,\R^3)|v=  0\ \text{on}\ \partial\Omega_D \}.
\]
For the bilinear form $\mathscr{B}_\Omega$ in \eqref{WeakEquation} using the ellipticity condition \eqref{eq:elasticity_tensor_ellipticity} we get for all \linebreak $v \in \mathbb{V}(\Omega)$:
\begin{align*}
\mathscr{B}_\Omega(v,v) & = \int\limits_\Omega \varepsilon(Dv):\sigma(Dv) dx\\
& \geq 2\mu \int\limits_\Omega  \varepsilon(Dv):\varepsilon(Dv) dx = 2\mu\|\varepsilon(Dv)\|_{0,\Omega}^2.
\end{align*}
Using this, we obtain
\begin{align}
2\mu \|\varepsilon(Du_n)\|_{0,\Omega_n}^2 \leq \mathscr{B}_{\Omega_n}(u_n,u_n) &= \int\limits_{\Omega_n} f \cdot u_n dx + \int\limits_{\partial(\Omega_n)_{N_\text{fixed}}} g \cdot u_n ds \\
&\leq (c_1 + c_2) \cdot \| u_n \|_{1,\Omega_n},
\end{align}
where the constant $c_1$ accounts for the bound of the integral over $f \cdot u_n$ and $c_2$ originates from the application of the trace theorem over the fixed Neumann boundary of the domain. Note that while $c_2$ depends on the domain under consideration obviously we can use the extension $\bar{u}_n$ of $u_n$ to $\widehat{\Omega}$. As we have
\[
\| \bar{u}_n \|_{1,\widehat{\Omega}} \leq K(\theta,l,r) \| \bar{u}_n \|_{1,\Omega_n} = K(\theta,l,r) \| u_n \|_{1,\Omega_n}
\]
the estimate holds by including the factor $K(\theta,l,r)$ in $c_2$. From Korn's inequality we can follow that there exists a $\beta > 0$ such that for all $v \in \mathbb{V}(\Omega)$ we have
\[
\|\varepsilon(v)\|_{0,\Omega}^2 \geq \beta \|v\|_{1,\Omega};
\]
furthermore, a result from \cite{Nit} guarantees that this $\beta$ can be uniformly bounded for all domains under consideration and we obtain that there exists a constant $c$ for all $\Omega_n$ such that
\[
\|u_n\|_{1,\Omega_n} \leq c.
\]
Due to Theorem~\ref{thm:Chenais_ii.1} the extension $\tilde{u}_n$ of $u_n$ to $\mathbb{R}^3$ is bounded and so is the extension to $\widehat{\Omega}$. Using this and Lemma~\ref{lem:pi_theta_h_r_is_relative_compact_and_closed} we obtain that there exists a subsequence of $\{\Omega_n,\tilde{u}_n\}_{n = 1}^\infty$ where $\Omega_n \rightarrow \Omega$ due to Lemma~\ref{lem:pi_theta_h_r_is_relative_compact_and_closed} and where $\tilde{u}_n$ converges weakly to some function $u \in (H^1(\widetilde{\Omega}))^3$.

It remains to show that this $u$ solves~\eqref{WeakEquation}. For this purpose we proceed as in the proof of Proposition IV.1 in \cite{Chen}. We have that
\[
\mathscr{B}_{\Omega_n}(u_n,v) = \int_{\Omega_n} f\cdot v\,dx+\int_{\partial(\Omega_n)_N} g\cdot v \,ds, \quad \forall v\in  H^1_{\partial (\Omega_n)_D}(\Omega_n).
\]
This is equivalent to
\[
\int_{\widehat{\Omega}} \chi(\Omega_n) \tr(\varepsilon(Du)\sigma(Dv)) \, dx = \int_{\widehat{\Omega}} \chi(\Omega_n) f\cdot v \, dx + \int_{\widehat{\Omega}} \chi(\partial(\Omega_n)_N) g \cdot v \, ds,
\]
$\quad \forall v\in  H^1_{\partial \Omega_D}(\Omega)$, where $\chi(\partial(\Omega_n)_N)$ denotes the characteristic function of $\partial(\Omega_n)_N$ as usual. We show the convergence of each of the integrals. For the first integral of the right-hand side we obviously have that for each $v \in L^2(\widehat{\Omega})$ we have
\[
|\chi(\Omega_n) v| \leq |v|
\]
and as the characteristic function converges a.e., we obtain that $\chi(\Omega_n) v \rightarrow \chi(\Omega) v$ and so we get
\[
\int_{\widehat{\Omega}} \chi(\Omega_n) f\cdot v \, dx \quad \rightarrow \int_{\widehat{\Omega}} \chi(\Omega) f\cdot v \, dx.
\]
As this is true for $v \in L^2(\widehat{\Omega},\R^3)$ it holds for $H^1_{\partial \Omega_D}(\Omega,\R^3)$, as well. For the partial derivatives in the integral on the left-hand side the same argument holds true. For the second integral on the right-hand side we can argue in the same manner, as the convergence of $\chi(\Omega_n)$ implies the convergence of the characteristic function of the boundary $\partial(\Omega_n)_N$. 
\end{proof}

We are now in the position to prove the main result of this work:

\begin{theorem}
Let $\partial\Omega_D$ and $\partial\Omega_N$ be defined as above and let $\{\Omega_n,u(\Omega_n)\}_{n=1}^\infty$ be a minimizing sequence of admissible domains in $\mathscr{O}^{\rm ad}$ and their corresponding solutions, i.e. $u(\Omega_n)$ solves~(\ref{WeakEquation}) in the domain $\Omega_n$ and
\begin{equation}
 \lim_{n\to\infty}J(\Omega_n,u(\Omega_n))=\inf_{\Omega\in \mathscr{O}^{\rm ad}}J(\Omega,u(\Omega)),
\end{equation}
where $J$ is defined as in \eqref{eq:Jbar}.
Moreover, assume that the crack size measure $\rho$ fulfils the non decreasing stress hazard property. Let $\tilde u^*$ and $\Omega^*\in\mathscr{O}^{\rm ad}$ be the limit points of a subsequence as defined in Lemma \ref{lem:convergence_of_domain_and_solution}. Then the restriction $u^*=u(\Omega^ *)$ of the weak limit $\tilde{u}^ *$ and $\Omega^ *$ solve the shape optimization problem \ref{defSOProblem}. Thus there exist shapes $\Omega^*\in \mathscr{O}^{\rm ad}$  that maximize the probability of survival \ref{defSurvProb}. This in particular applies to the Weibull model for $m>0$. The above statements also remain true for the volume constraint shape optimization problem, where $\mathscr{O}^{\rm ad}$ is replaced by $\mathscr{O}_V^{\rm ad}$, the admissible shapes of volume $V$, provided this set is not empty.
\end{theorem}
\begin{proof}
The function $h$ given by \eqref{hexplizit} is obviously non-negative, as the integrand is non-negative. Proposition~\ref{propCov} gives convexity of $h$ and Lemma~\ref{lem:hconti} its continuity. Furthermore, the strong convergence of the domains and the weak convergence of the corresponding solution is guaranteed by Lemma~\ref{lem:convergence_of_domain_and_solution}, so all requirements of Theorem~\ref{thm:fuji_thm2.1} are fulfiled and the assertion follows. By Proposition \ref{PropWeibull} and Corollary \ref{cor:Weibull_local_failure_intensity_function}, the Weibull model is a special case. 
\end{proof}

\section{Conclusions}
In this paper we have proven the existence of shapes with minimal failure probability for the case of ceramic components with given loads with and without volume constraint. A number of further questions naturally arise.

First of all the uniqueness of the optimal solution has not been investigated. As $\mathscr{O}^{\rm ad}$ does not have a linear structure, convexity properties of the functional $J(\Omega,u(\Omega))$ are not easily defined. Here the understanding of $\mathscr{O}^{\rm ad}$ as a infinite dimensional manifold \cite{Schul} and the study of convexity on manifolds \cite{Udr} might be of interest.

One of the attractive features of the objective functionals that originate from the probabilistic analysis is that they fit quite nicely into the framework of shape calculus \cite{SZ} and it is a natural requirement to consider algorithms for the actual maximization of survival probabilities that are based on shape gradients. In fact it is not difficult to follow the calculations of \cite[Chapter 3]{SZ} in order to show that on the level of formal calculations
\begin{equation}
dJ(\Omega,u(\Omega))[\mathscr{V}]:=\frac{d}{dt}J(\Omega_t,u(\Omega_t))\restriction_{T=0}=\int_\Omega \nabla h(Du):Du'[\mathscr{V}]\, dx
\end{equation}
is the shape derivative of $dJ(\Omega,u(\Omega))$ with respect to the vector field $\mathscr{V}:\widehat \Omega\to\R^3$. $\Omega_t$ here stands for the image of $\Omega$ under the flow generated by $\mathscr{V}$. $u'[\mathscr{V}]$ is the shape derivative of $u$ with respect to $\mathscr{V}$ fulfils the PDE given in \cite[Theorem 3.11]{SZ}. The first order optimallity conditions can then be written as $dJ(\Omega,u(\Omega))[\mathscr{V}]=0$ $\forall \mathscr V$. The mathematical details however are more subtle and go beyond the scope of this article. Notably \cite[Theorem 3.11]{SZ} obviously requires more regularity of $u(\Omega)$ that it is provided by weak solutions that have been used here.

A second interesting aspect is the question, if the uniform cone condition in the definition of $\mathscr{O}^ {\rm ad}$ could be relaxed. It seems to us that inward corners due to the stress concentration that occurs at the tip, will be effectively penalized by high failure probabilities. Outward corners do not carry stress and thus are not helpful either, at least if a volume constraint is active. Rough, fractal boundaries therefore do not seem to be preferred by the objective functional. So the somewhat artificial geometric constraints that are hidden in the constants $\theta,l,r$ of the cone property might well turn out to be redundant, if the analysis is carried further.

\noindent {\bf Acknowledgements:} We would like to thank Patricia H\"ulsmeier and Christoph Ziegeler for making their Ph.D. Theses available to us. We are grateful to Rolf Krause from ICS Lugano for interesting discussions. We also thank the referees for reading the submitted article very carefully and providing many suggestions for improvement.

\bibliographystyle{jota} 
\bibliography{keramik}

\newpage

\noindent {\sc Matthias Bolten and Hanno Gottschalk} \\ 
Department of
Mathematics and Science, \\
Bergische Universit\"at
Wuppertal,\\ {\tt bolten@math.uni-wuppertal.de},\\
{\tt hanno.gottschalk@uni-wuppertal.de}

\vspace{1cm}

\noindent {\sc Sebastian Schmitz }\\
 Institute of Computational
Science,\\
 Universit\`a della Svizzera Italiana,\\
  Lugano,\\
{\tt sebastian.schmitz@usi.ch}

\end{document}